\documentclass[12pt]{amsart}

\title[cocycles over higher-rank abelian actions]
      {cocycles over higher-rank abelian actions on quotients of semisimple Lie groups}
\author{Felipe A. Ram\'\i rez}
\thanks{The author was supported by NSF RTG 0602191}

\usepackage{amscd}
\usepackage{verbatim}
\usepackage[all]{xy}
\usepackage{color}

\newtheorem*{theoremA}{Theorem A}
\newtheorem*{theoremA'}{Theorem A'}
\newtheorem*{theoremB'}{Theorem B'}
\newtheorem*{theoremC'}{Theorem C'}
\newtheorem*{theoremD'}{Theorem D'}
\newtheorem{theorem}{Theorem}[section]
\newtheorem{lemma}[theorem]{Lemma}
\newtheorem{proposition}[theorem]{Proposition}
\newtheorem{corollary}[theorem]{Corollary}

\theoremstyle{definition}

\theoremstyle{remark}
\newtheorem{remark}[theorem]{Remark}

\theoremstyle{plain}
\newtheorem{result}{Theorem}

\newcommand{\NN}{\mathbb{N}}
\newcommand{\RR}{\mathbb{R}}

\newcommand{\CC}{\mathbb{C}}
\newcommand{\Cinf}{C^{\infty}}

\newcommand{\ZZ}{\mathbb{Z}}

\newcommand{\U}{\mathcal{U}}
\newcommand{\T}{\mathcal{T}}
\newcommand{\X}{\mathcal{X}}
\newcommand{\Y}{\mathcal{Y}}
\newcommand{\Z}{\mathcal{Z}}
\newcommand{\V}{\mathcal{V}}
\newcommand{\W}{\mathcal{W}}

\newcommand{\g}{\mathfrak{g}}
\newcommand{\h}{\mathfrak{h}}
\newcommand{\z}{\mathfrak{z}}
\newcommand{\LL}{\mathfrak{L}}

\newcommand{\uu}{\mathfrak{u}}

\newcommand{\Sl}{\mathfrak{sl}}
\newcommand{\SL}{\mathrm{SL}}
\newcommand{\PSL}{\mathrm{PSL}}

\newcommand{\HH}{\mathcal{H}}
\newcommand{\II}{\mathcal{I}}
\newcommand{\ex}{\mathrm{exp}}
\newcommand{\dist}{\mathrm{dist}}

\DeclareMathOperator{\Ad}{Ad}
\DeclareMathOperator{\ad}{ad}

\def\a{{\alpha}}
\def\b{{\beta}}

\def\e{{\epsilon}}
\def\s{{\sigma}}
\def\t{{\tau}}

\def\l{{\lambda}}
\def\L{{\Lambda}}
\def\o{{\omega}}
\def\G{{\Gamma}}

\begin{document}

\begin{abstract}
We study actions by higher-rank abelian groups on quotients of semisimple Lie groups with finite center.  First, we consider actions arising from the flows of two commuting elements of the Lie algebra---one nilpotent, and the other semisimple.  Second, we consider actions from two commuting unipotent flows that come from an embedded copy of $\overline{\SL(2,\RR)}^{k} \times \overline{\SL(2,\RR)}^{l}$.  In both cases we show that any smooth $\RR$-valued cocycle over the action is cohomologous to a constant cocycle via a smooth transfer function.  These build on results of D. Mieczkowski, where the same is shown for actions on $(\SL(2,\RR) \times \SL(2,\RR))/\G$.
\end{abstract}

\maketitle

\section{Introduction} \label{intro}

This work is concerned with smooth $\RR$-valued cocycles over actions by higher-rank subgroups on quotients of semisimple Lie groups.  The goal is to show that, in the cases we consider, all such cocycles are smoothly cohomologous to constant cocycles---cocycles whose values only depend on the acting group.  Our results are in the vein of work done by A. Katok and R. Spatzier \cite{KS2}, L. Flaminio and G. Forni \cite{FF}, and D. Mieczkowski \cite{M1, M2}.  Our work relies heavily on theirs.

In \cite{KS2}, Katok and Spatzier showed that smooth cocycles over Anosov actions by higher-rank abelian groups are cohomologically constant, via smooth transfer functions.  This is in contrast to the rank one situation, where Livsic showed that there is an infinite-dimensional space of obstructions to solving the cohomology equation for a hyperbolic action by $\RR$ or $\ZZ$.  For both results, the stable and unstable foliations of the space play a central role.  In particular, the regularity of transfer functions is achieved by studying their behavior along leaves of these foliations.  We will employ similar methods to show that our transfer functions are smooth.

In \cite{FF}, Flaminio and Forni characterized the obstructions to solving the cohomology equation for horocycle flows on quotients of $\PSL(2,\RR)$.  Suppose $\HH$ is a unitary representation of $\PSL(2,\RR)$.  If the Casimir operator on $\HH$ has a spectral gap, then the obstructions are given by distributions that are invariant under the flow of $\U =  \big( \begin{smallmatrix} 0 & 1 \\ 0 & 0 \end{smallmatrix}\big)$.  That is, for a smooth vector $f \in C^{\infty}(\HH)$, if $D(f)=0$ for every $\U$-invariant distribution $D$, then there exists a smooth vector $P \in C^{\infty}(\HH)$ such that $\U P=f$.  (In fact, Flaminio and Forni showed this for $f \in W^{s}(\HH)$, the Sobolev space of order $s$; in this case, $P$ comes with some loss of regularity.)

In \cite{M1}, Mieczkowski showed that for smooth cocycles over certain actions on $(\SL(2,\RR) \times \SL(2,\RR))/\G$, where $\G$ is an irreducible lattice, the obstructions to solving the cohomology equation vanish, and one can find a smooth solution.  He considered actions by the subgroups
\[
A  = 
\left\{\left( \begin{array}{cc}
1 & r \\
0 & 1
\end{array}\right) \times
\left( \begin{array}{cc}
e^{t/2} & 0 \\
0 & e^{-t/2}
\end{array}\right) | r,t \in \RR \right\}
\]
and
\[
U  = 
\left\{\left( \begin{array}{cc}
1 & r \\
0 & 1
\end{array}\right) \times
\left( \begin{array}{cc}
1 & s \\
0 & 1
\end{array}\right) | r,s \in \RR \right\}.
\]
(His results also hold for cocycles in a Sobolev space.  Again, the solutions to the cohomology equation come with some loss of Sobolev order.)  Like the results of \cite{FF}, this result is achieved in any unitary representation of $\SL(2,\RR) \times \SL(2,\RR)$, provided the Casimir operator for that representation has a spectral gap.  He then applies this to the left-regular representation on $L^{2}((\SL(2,\RR) \times \SL(2,\RR))/\G)$.

Our first two results are similar to those of Mieczkowski's, replacing $\SL(2,\RR) \times \SL(2,\RR)$ with any noncompact simple Lie group $G$ with finite center.  First, we consider actions on a compact $G/\G$ arising from two commuting flows---one along a nilpotent element of the Lie algebra, and the other along a commuting semisimple element.  We prove the following

\begin{theoremA'} \label{aprime}
Let $G$ be a noncompact simple Lie group with finite center and Lie algebra $\g$, and let $\G \subset G$ be a cocompact lattice.  Let $\U \in \g$ be nilpotent and $\X \in \g$ be semisimple such that the eigenvalues of $\ad(\X)$ are not all purely imaginary, and $[\U,\X]=0$.  Then a smooth $\RR$-valued cocycle over the action by the flows of $\U$ and $\X$ on $G/\G$ is smoothly cohomologous to a constant cocycle.
\end{theoremA'}

Second, we consider smooth cocycles over unipotent actions.  We show that, in the case where the Lie group in question admits an embedding of $\overline{\SL(2,\RR)}^{k} \times \overline{\SL(2,\RR)}^{l}$ (where $\overline{\SL(2,\RR)}^{m}$ denotes an $m$-sheeted cover of $\SL(2,\RR)$), and the acting unipotent group contains the unipotent elements of $\overline{\SL(2,\RR)}^{k} \times \overline{\SL(2,\RR)}^{l}$, such cocycles are cohomologous to constant cocycles, via smooth transfer functions.  For the unipotent case, we do not require $\G \subset G$ to be cocompact.  We have

\begin{theoremB'} \label{bprime}
Let $G$ be a noncompact simple Lie group with finite center and Lie algebra $\g$, and suppose $\G \subset G$ is a lattice.  Suppose $\overline{\SL(2,\RR)}^{k} \times \overline{\SL(2,\RR)}^{l}$ embeds in $G$.  Consider $\U_{1} =  \big( \begin{smallmatrix} 0 & 1 \\ 0 & 0 \end{smallmatrix}\big) \times (0)$ and $\U_{2} = (0) \times \big( \begin{smallmatrix} 0 & 1 \\ 0 & 0 \end{smallmatrix}\big) \in \Sl(2,\RR) \times \Sl(2,\RR) \subset \g$.  Then a smooth $\RR$-valued cocycle over the action by the flows of $\U_{1}$ and $\U_{2}$ on $G/\G$ is smoothly cohomologous to a constant cocycle.
\end{theoremB'}

Using Theorem B', we prove

\begin{theoremC'} \label{cprime}
Let $G$, $\G$ and $\U_{1}$, $\U_{2}$ be as in Theorem B'.  Let $U \subset G$ be the rank-$2$ abelian subgroup generated by $\U_{1}$ and $\U_{2}$, and let $V \subset G$ be the maximal unipotent subroup containing $U$.  Then a smooth $\RR$-valued cocycle over the $V$-action on $G/\G$ is smoothly cohomologous to a constant cocycle.
\end{theoremC'}

Theorems A', B', and C' are stated in more generality in Section \ref{results}.  

\section{Background, Definitions, and Statement of Results}
In this section we give basic definitions and background.  We also state our main results.

\subsection{Cocycles}\label{cocycles}
The following definitions are standard, and can all be found in \cite{KR} and \cite{M2}.  For a survey of the uses of cocycles in dynamics, see \cite{KR}.

For a measurable action of a group $H$ on a measure space $(X, \mu)$, a $G$-\emph{valued degree $1$ cocycle} is defined to be a measurable map $\a: H \times X \rightarrow G$ satisfying
\begin{equation} \label{coc}
\a(h_{1}h_{2}, x) = \a(h_{1}, h_{2}x)\a(h_{2}, x),
\end{equation}
where $G$ is a group.  (It should be noted that all of the cocycles in this paper will be $\RR$-valued degree $1$ cocycles, and so we will often refer to them simply as cocycles, or $\RR$-valued cocycles, since there is no risk of confusion.)  Equation \eqref{coc} is called the \emph{cocycle identity}.  Occasionally, it will be convenient to think of an $\RR$-valued cocycle as being a map $\a: H \rightarrow \mathcal{F}(X)$ where $\mathcal{F}(X)$ denotes the measurable functions on $X$.  In this case, the cocycle identity is
\[ \a(h_{1}h_{2})(x) = \a(h_{1})(h_{2}x) + \a(h_{2})(x). \]
A $G$-cocycle whose image is the identity element in $G$ is called a \emph{trivial cocycle}.  A homomorphism $\phi:H \rightarrow G$ satisfies the cocycle identity by setting $\phi(h,x)=\phi(h)$, and is called a \emph{constant cocycle}.

Two $G$-cocycles $\a$ and $\b$ are said to be \emph{cohomologous} if there exists a measurable map $P:X \rightarrow G$ such that
\begin{equation} \label{coh}
\b(h,x) = P(hx)^{-1}\a(h,x)P(x).
\end{equation}
The map $P$ is referred to as a \emph{transfer function}, and \eqref{coh} is called the \emph{cohomology equation}.  (Notice that if the group $G$ is abelian, and $P$ satisfies equation \eqref{coh}, then so does $g \cdot P$ for any fixed $g \in G$.)  We say that a cocycle is a \emph{coboundary} if it is cohomologous to the trivial cocycle.  It is an \emph{almost coboundary} if it is cohomologous to a constant cocycle.  

This paper will be concerned exclusively with smooth $\RR$-valued cocycles over group actions on smooth manifolds.  Specifically, the acting group will be a connected Lie subgroup $H$ of a connected simple Lie group $G$, and the space $X$ will be $X=G/\G$, where $\G \subset G$ is a lattice.  For $\a$ to be a \emph{smooth cocycle}, we require that it be a smooth map in the usual sense, and that $\a(h,\_)$ be a smooth vector in $L^{2}(G/\G)$ for all $h \in H$.  That is, $\a(h,\_) \in C^{\infty}(L^{2}(G/\G))$.

In this context we can define the \emph{infinitesimal generator} of the cocycle $\a$ by $\o(\V) = \frac{d}{dt}\a(\exp t \V)|_{t=0}$.  The cocycle identity implies that $\o$ is a closed $1$-form on the $H$-orbits in $X$.  The cohomology equation then becomes $\o = \eta - dP$, where $P$ is the transfer function, and $\eta$ is another smooth cocycle.  Therefore, in this context, a cocycle $\a$ is cohomologically trivial if its associated $1$-form $\o$ is exact.  It should also be noted that if the cocycle $\a$ is cohomological to a constant cocycle, then that constant cocycle is given by
\[ c(h) = \int_{G/\G}\a(h,g)dg_{\G}. \]

Given a closed $1$-form on the $H$-orbit foliation of $X$, one can recover the cocycle $\a$ by $\a(\exp \V) = \int_{0}^{1}\o(\V) \cdot \exp t \V dt$.  Thus, the problem of determining which cocycles are cohomologically trivial can be translated to the problem of finding which closed $1$-forms on the $H$-orbits of $X$ are exact.  In fact, this point of view is the most useful for our purposes.

\subsection{Some useful theorems and definitions}

\subsubsection{Representations and Sobolev spaces}

One of the main tools we use to study the cohomology equation is the representation theory of semisimple Lie groups.  The following are some basic facts and definitions that can be found in \cite{Warner} and \cite{Knapp}.

Given a unitary representation $\pi: G \times \HH \rightarrow \HH$, one says that $v \in \HH$ is a \emph{smooth vector} if the map $g \mapsto \pi(g)v$ is smooth in the usual sense.  For the left-regular representation of a Lie group $G$ on $L^{2}(G/\G)$, where $\G \subset G$ is a lattice, a smooth vector is a smooth function $f \in L^{2}(G/\G)$ such that $\V^{k}f \in L^{2}(G/\G)$ for all $\V \in \mathrm{Lie}(G)$ and $k \in \NN$.  In this case, we write $f \in C^{\infty}(L^{2}(G/\G))$.  If $\G$ is cocompact, then the smooth vectors are exactly the smooth functions on $G/\G$.

It is often useful to consider a less restrictive subspace of the unitary representation $\HH$ of $G$, called the \emph{Sobolev space} of order $s \in \ZZ_{+}$, and denoted $W^{s}(\HH)$.  It is defined as the maximal domain of the operator $(I - \Delta)^{s/2}$, where $\Delta$ denotes the Laplacian from $G$.  $W^{s}(\HH)$ is a Hilbert space with inner product defined by
\[
<f,g>_{s} = <(I - \Delta)^{s} f, g>_{\HH}.
\]
Sobolev spaces of representations of $\SL(2,\RR)$ are of particular importance to this work, insofar as it is necessary to consider them in order to apply Theorem \ref{FF} \cite{FF}.

We recall the following theorem of Kolmogorov-Mautner \cite{Starkov}, which will allow us to restrict our attention to \emph{irreducible} unitary representations for much of our study.

\begin{theorem}[Kolmogorov-Mautner] \label{directintegral}
Given any any unitary representation $\pi$ of a locally compact second countable group $G$ in a separable Hilbert space $\HH$, there exists a Lebesgue-Stieltjes measure $d \mu$ on $\RR$ such that $\HH$ is the direct integral $\HH = \int_{\RR}{\HH_{\mu}d \mu}$ of Hilbert spaces $\HH_{\mu}$ with unitary representations $\pi_{\mu}$ of the group $G$ on $\HH_{\mu}$, where $\pi(g)f = \int_{\RR}{\pi_{\mu}(g)f_{\mu}d \mu}$.  For $d \mu$-almost all $\mu \in \RR$, the representation $\pi_{\mu}$ is irreducible.
\end{theorem}

\subsubsection{Ergodicity}
Let $G$ be a Lie group, $\G \subset G$ a lattice. The flow along $\V \in \g := \mathrm{Lie}(G)$ on $G/\G$, which we will often denote $\phi_{t}^{\V}$, is said to be \emph{ergodic} if every $\V$-invariant measurable (with respect to Haar measure) subset of $G/\G$ is either a nullset or has full measure.  If $G$ is noncompact and semisimple, and $\G \subset G$ is irreducible, then the only elements of $\g$ whose flows are \emph{not} ergodic are semisimple elements $\V \in \g$ such that $\ad(\V)$ has purely imaginary eigenvalues \cite{Anosov}.  This follows from the Howe-Moore ergodicity theorem, which we quote from \cite{FK}.
\begin{theorem}[Howe-Moore] \label{HoweMoore}
Let $G$ be a noncompact simple Lie group with finite center and let $\G \subset G$ be a lattice in $G$.  Then any closed noncompact subgroup $H$ of $G$ acts ergodically on $G/\G$ by left translations.
\end{theorem}
We will make extensive use of Theorem \ref{HoweMoore} throughout this work.

\subsubsection{Partially hyperbolic flows}
A flow $\phi_{t}$ on a smooth manifold $M$ is \emph{partially hyperbolic} if there exists a splitting
\[
TM = E^{-} \oplus E^{0} \oplus E^{+},
\]
and constants $A, B, \L_{-}, \L_{+} \in \RR_{+}$ such that
\[
\left\| d\phi_{t}(\V) \right\| \leq A \cdot e^{-t \L_{-}} \cdot \left\| \V \right\|
\]
for all $\V \in E^{-}$ and $t >0$, and
\[
\left\| d\phi_{-t}(\W) \right\| \leq B \cdot e^{-t \L_{+}} \cdot \left\| \W \right\|
\]
for all $\W \in E^{+}$ and $t > 0$.  $E^{-}$ and $E^{+}$ are called the \emph{stable} and \emph{unstable} distributions for the flow $\phi_{t}$.  These integrate to the stable and unstable foliations, $W^{-}$ and $W^{+}$.

If $G$ is a noncompact semisimple Lie group, and $\G \subset G$ is an irreducible lattice, then the flow $\phi_{t}^{\X}$ on $G/\G$ is partially hyperbolic for any semisimple $\X \in \g:=\mathrm{Lie}(G)$ whose roots are not all purely imaginary.  The distributions $E^{-}$ and $E^{+}$ are invariant under translation on the right by group elements, and so we can identify them with subspaces of the Lie algebra $\g$ of right-invariant vector fields on $G$.  We will often make this identification implicitly; that is, we will write
\[
\g = E^{-} \oplus E^{0} \oplus E^{+},
\]
and refer to elements of the distributions $E^{\pm}$ as though they are members of the Lie algebra $\g$.  It should be understood that we are really referring to the elements' images in $\g$ under this identification by taking right-translates.

\subsubsection{Smooth functions}
The following result of A. Katok and R. J. Spatzier \cite{KS} will allow us to bootstrap the regularity of certain functions (those that have derivatives in a spanning set of directions) to smoothness.
\begin{theorem}[Katok-Spatzier] \label{KaSp}
let $D_{1}, \ldots, D_{k}$ be $C^{\infty}$ plane fields on a manifold $M$ such that their sum $\sum_{i=1}^{k}D_{i}$ is totally non-integrable and satisfies the following condition:
For each $j$, the dimension of the space spanned by the commutators of length at most $j$ at each point is constant in a neighborhood.
Let $P$ be a distribution on $M$.  Assume that for any positive integer $p$ and $C^{\infty}$ vector field $X$ tangent to any $D_{j}$, the $p^{\mathrm{th}}$ partial derivative $X^{p}(P)$ exists as a continuous or local $L^{2}$ function.  Then $P$ is $C^{\infty}$ on $M$.
\end{theorem}

\subsection{Main results} \label{results}
Let $G = G_{1} \times \cdots \times G_{k}$ be a product of noncompact simple Lie groups, $\g = \g_{1} \oplus \cdots \oplus \g_{k}$ its Lie algebra, where $\g_{i} := \mathrm{Lie}(G_{i})$ for $i = 1, \ldots, k$.  We prove the following theorems.

\begin{result} \label{nilpotentsemisimple}
Suppose $\G \subset G$ is a cocompact irreducible lattice.  Suppose $\U \in \g$ is nilpotent and $\X \in \g$ is semisimple such that $[\U,\X] = 0$ and each $\g_{i}$ contains stable and unstable vectors for the flow $\phi_{t}^{\X}$.  Then any smooth $\RR$-valued cocycle over the action by $\RR^{2}$ on $G/\G$ defined by the flows $\phi_{t}^{\U}$ and $\phi_{t}^{\X}$ is cohomologous to a constant cocycle, via a smooth transfer function.
\end{result}

\begin{result} \label{nilpotentnilpotent}
Suppose $G$ admits an embedding of $\overline{\SL(2,\RR)}^{k} \times \overline{\SL(2,\RR)}^{l}$, and $\G \subset G$ is an irreducible lattice.  Consider $\U_{1} =  \big( \begin{smallmatrix} 0 & 1 \\ 0 & 0 \end{smallmatrix}\big) \times (0)$ and $\U_{2} = (0) \times \big( \begin{smallmatrix} 0 & 1 \\ 0 & 0 \end{smallmatrix}\big) \in \Sl(2,\RR) \times \Sl(2,\RR) \subset \g$.  If the projection of $\U_{1} + \U_{2}$ to $\g_{i}$ is nonzero for all $i = 1, \ldots, k$, then any smooth $\RR$-valued cocycle over the action by $\RR^{2}$ on $G/\G$ defined by the flows $\phi_{t}^{\U_{1}}$ and $\phi_{t}^{\U_{2}}$ is cohomologous to a constant cocycle, via a smooth transfer function.
\end{result}

\begin{remark}
Observe that in Theorem \ref{nilpotentsemisimple}, we require the lattice to be cocompact, whereas in Theorem \ref{nilpotentnilpotent} we do not.  For Theorem \ref{nilpotentsemisimple} we only use cocompactness in Section \ref{smoothnessofP} to show that the transfer functions are smooth.  In the proof of Theorem \ref{nilpotentnilpotent}, we use a different method---one that does not require compactness of the space---to establish the smoothness of transfer functions.
\end{remark}

\begin{remark}
Notice that Theorems A' and B' in the Introduction (Section \ref{intro}) are the same theorems as above, in the case where $G$ is simple.
\end{remark}

In \cite{M1} and \cite{M2}, Mieczkowski proved Theorems \ref{nilpotentsemisimple} and \ref{nilpotentnilpotent} for the case where $G = \SL(2,\RR) \times \SL(2,\RR)$.  The result was achieved using tools from the unitary representation theory of $\SL(2,\RR)$.  Mieczkowski's results are essential to our work.

Given an embedding $\overline{\SL(2,\RR)}^{k} \times \overline{\SL(2,\RR)}^{l} \hookrightarrow G$, one can consider a maximal unipotent subgroup $V \subset G$ containing the unipotent elements of the embedded $\overline{\SL(2,\RR)}^{k} \times \overline{\SL(2,\RR)}^{l}$ obtained by exponentiating $\U_{1}$ and $\U_{2}$.  Using Theorem \ref{nilpotentnilpotent}, we prove

\begin{result} \label{maximalunipotent}
Let $G$, $\G$ and $\U_{1}$, $\U_{2}$ be as in Theorem \ref{nilpotentnilpotent}.  Let $U \subset G$ be the rank-$2$ abelian subgroup generated by $\U_{1}$ and $\U_{2}$, and let $V \subset G$ be the maximal unipotent subroup containing $U$.  Then a smooth $\RR$-valued cocycle over the $V$-action on $G/\G$ is cohomologous to a constant cocycle via a smooth transfer function $P \in \Cinf(G/\G)$.
\end{result}

As an easy application of \ref{maximalunipotent}, we obtain the following.

\begin{corollary}
Any smooth $\RR$-valued cocycle over the action by the upper triangular group $V \subset \SL(n,\RR)$ (with $n>3$) on $\SL(n,\RR)/\G$ is smoothly cohomologous to a constant cocycle.
\end{corollary}

\begin{proof}
The result follows by considering the embedding of $\SL(2,\RR) \times \SL(2,\RR)$ into the first two diagonal $2 \times 2$ blocks of $\SL(n,\RR)$, and applying Theorem \ref{maximalunipotent}.
\end{proof}

\section{Proof of Theorem \ref{nilpotentsemisimple}}
\subsection{Strategy}
As in the statement of Theorem \ref{nilpotentsemisimple}, we let $G = G_{1} \times \cdots \times G_{k}$ be a product of noncompact simple Lie groups with finite center, and $\G \subset G$ a cocompact irreducible lattice.  Suppose $\U \in \g$ is nilpotent and $\X \in \g$ is a semisimple element such that $\left[\U,\X \right]=0$ and such that there are stable and unstable vectors in each $\g_{i} := \mathrm{Lie}(G_{i})$.  Then the commuting flows $\phi_{t}^{\U}$ and $\phi_{t}^{\X}$ of $\U$ and $\X$ on $G/\G$ form a group action by $\RR^{2}$.  Suppose $\a$ is a smooth cocycle over this action.  Then its infinitesimal generator $\o$ is determined by the smooth functions
\[
f=\o(\U) \quad \textrm{and} \quad g=\o(\X),
\]
and these satisfy the relation $\U g = \X f$.  Now, finding a smooth solution to the cohomology equation for $\a$ is equivalent to finding a smooth function $P:G/\G \rightarrow \RR$ such that
\[
\U P=f \quad \textrm{and} \quad \X P=g.
\]

Our strategy is to choose a subalgebra $\h \subset \g$ containing $\U$ and $\X$, and consider its corresponding subgroup $H \subset G$.  We have the left-regular unitary representation of $H$ on $L^{2}(G/\G)$, so there is a direct integral decomposition
\[
L^{2}(G/\G) = \int_{\oplus}{\HH_{\nu}ds(\nu)},
\]
where $ds$-almost all $\HH_{\nu}$ are irreducible.  Naturally, the corresponding decomposition of $f \in L^{2}(G/\G)$ is denoted
\[
f = \int_{\oplus}{f_{\nu}ds(\nu)}, \quad f_{\nu} \in \HH_{\nu}.
\]
(This decomposition also holds for the Sobolev spaces, $W^{s}(L^{2}(G/\G))$.)  The idea will be to choose an $\h$ whose representations are well-enough understood that we can find solutions $P_{\nu}$ to the cohomology equation in each irreducible $\HH_{\nu}$.  This, together with estimates on the Sobolev norms of the $P_{\nu}$, will allow us to glue these solutions together to get a global solution $P \in L^{2}(G/\G)$.

Next, we must show that the solution $P$ is smooth on $G/\G$.  For this, we consider the stable and unstable submanifolds of $G/\G$ with respect to the flow $\phi_{t}^{\X}$ along the semisimple element $\X \in \g$.  A Livsic type argument will show that $P$ is smooth along these foliations.  (It is worth noting that this is the only place where the cocompactness of $\G \subset G$ is used.)  Then, since these directions span $\g$ as a Lie algebra, we can use Theorem \ref{KaSp} to show that $P$ is smooth on $G/\G$.  

The following sections are devoted to proving these claims.

\subsection{Restatement of the problem}
We have a noncompact semisimple Lie group $G$ with finite center, and $\G \subset G$ a cocompact irreducible lattice.  We have a nilpotent element $\U \in \g := \mathrm{Lie}(G)$ and a commuting semisimple element $\X \in \g$.  The flows of $\U$ and $\X$ on $G/\G$ can be seen as a group action by the subgroup $A \subset G$ corresponding to the Lie subalgebra generated by $\U$ and $\X$.  We are given a cocycle
\[
\a : A \times G/\G \rightarrow \RR
\]
that is smooth in the sense that it is a smooth map, and $\a(a,\_) \in C^{\infty}(L^{2}(G/\G))$ for all $a \in A$.  Defining the infinitesimal generator $\o$ of $\a$ as in Section \ref{cocycles}, one sees that $\a$ is completely determined by the functions $f = \o(\U), g = \o(\X) \in C^{\infty}(L^{2}(G/\G))$.  A simple calculation shows that the cocycle identity is now the relation $\U g = \X f$.  

From the cohomology equation one sees that if $\a$ is cohomologous to a constant cocycle $c: A \rightarrow \RR$, then $c$ is determined by
\[
c(a) = \int_{G/\G}\a(a,g)dg_{\G}.
\]
Thus, showing that all smooth cocycles are cohomologically constant is equivalent to showing that all smooth cocycles which integrate to $0$ are cohomologically trivial.  We will assume then that
\[
\int_{G/\G}\a(a,g)dg_{\G} = 0
\]
for all $a \in A$.

In terms of the infinitesimal generator $\o$, showing the cocycle is cohomologically trivial is equivalent to showing that $\o = dP$ for some function $P$.  We now recast Theorem \ref{nilpotentsemisimple} in terms of the functions $f$ and $g$ .

\begin{theoremA}
Suppose $G$ is a noncompact semisimple Lie group with finite center, and $\G \subset G$ is a cocompact irreducible lattice.  Suppose $\U \in \g$ is nilpotent and $\X \in \g$ is a semisimple element such that $[\U,\X]= 0$, and assume the Lie algebra of each factor of $G$ contains (un)stable vectors for the flow $\phi_{t}^{\X}$.  Suppose $f, g \in C^{\infty}(L^{2}(G/\G))$ satisfy $\U g = \X f$ and both $f$ and $g$ integrate to $0$.  Then there exists $P \in C^{\infty}(G/\G)$ such that $\U P = f$ and $\X P = g$.
\end{theoremA}

In the next section we introduce a subalgebra $\h \subset \g$ containing $\U$ and $\X$ whose representations will be useful to our treatment of the problem.

\subsection{Defining a useful subalgebra $\h$} \label{definingh}

By the following generalization of the Ja\-cobson-Morozov Lemma \cite{Starkov}, we can find a subalgebra $\h_{1} \in \g$ such that $\h_{1} \cong \Sl(2,\RR)$, $\U =  \big( \begin{smallmatrix} 0 & 1 \\ 0 & 0 \end{smallmatrix}\big) \in \Sl(2,\RR)$, and $\left[\X,\Sl(2,\RR) \right]=0$.

\begin{theorem}[Jacobson-Morozov]\label{JM}
Let $\U$ be a nilpotent element in a semisimple Lie algebra  $\g$, commuting with a semisimple element $\X \in \g$.  Then there exist a semisimple element $\Y \in \g$ and a nilpotent element $\V \in \g$ such that $\left[\Y,\U \right] = \U$, $\left[\Y,\V \right] = -\V$, and $\left[\U,\V \right] = \Y$, where $\Y$ and $\V$ commute with $\X$.
\end{theorem}

Now we can consider the subalgebra $\h := \h_{1} \times \h_{2} = \Sl(2,\RR) \times \RR \X$.  The subgroup $H \subset G$ corresponding to $\h$ is a product, $H = H_{1} \times H_{2}$ where $H_{1} = \overline{\SL(2,\RR)}^{k}$ is a $k$-sheeted cover of $\SL(2,\RR)$, and $H_{2} = \RR_{+}$.  The advantage of this is that the unitary representations of $H$ are easy to work with.  Our ultimate goal is to find $P \in C^{\infty}(G/\G)$ such that $\U P = f$ and $\X P = g$.  The first step toward achieving this is to prove the following lemma and apply it to the left-regular representation of $H$ on $L^{2}(G/\G)$.

\begin{lemma} \label{globalsolution}
Suppose $\HH$ is a unitary representation of $H_{1} \times H_{2}$, and suppose there is a spectral gap for the Casimir operator from $H_{1}$.  If $f$, $g \in C^{\infty}(\HH)$ satisfy $\U g = \X f$, then there exists $P \in \HH$ satisfying $\U P = f$.
\end{lemma}

The next few sections will be devoted to proving Lemma \ref{globalsolution}.  The full proof is stated in Section \ref{globalsolutionsection}.  First, we will summarize some of the details of the representation theory of $\h$.

\subsection{Representations of $\h$}\label{representations}
The subgroup of $G$ corresponding to the subalgebra $\h \subset \g$ is $H=H_{1} \times H_{2}$, where
$\h_{i}$ is the Lie algebra $H_{i}$.  Irreducible unitary representations of $H$ are of the form $\HH_{\mu} \otimes \HH_{\theta}$, where $\HH_{\mu}$ is an irreducible unitary representation of $H_{1}$ and $\HH_{\theta}$ is an irreducible unitary representation of $H_{2}$.  The subscripts $\mu$ and $\theta$ are explained in the following subsections.

\subsubsection{Representations of $\h_{1} = \Sl(2,\RR)$}
We will be concerned with the irreducible unitarizable representations of $\Sl(2,\RR)$; that is, those representations that arise as the derivatives of irreducible unitary representations of some Lie group whose Lie algebra is $\Sl(2,\RR)$ (in our study, this Lie group is being denoted $H_{1}$).  In fact, all such representations can be realized from irreducible unitary representations of some finite cover of $\SL(2,\RR)$.  In turn, all of these are unitarily equivalent to irreducible representations of $\SL(2,\RR)$, itself \cite{HT}.

We fix the following generators for $\Sl(2,\RR)$.
\[
\X = \left( \begin{array}{cc}
1/2 & 0 \\
0 & -1/2
\end{array}\right), \quad
\Y = \left( \begin{array}{cc}
0 & -1/2 \\
-1/2 & 0
\end{array}\right), \quad
\Theta = \left( \begin{array}{cc}
0 & 1/2 \\
-1/2 & 0
\end{array}\right).
\]
Then we have the Laplacian operator defined by $\Delta = \X^{2} + \Y^{2} + \Theta^{2}$, and the Casimir operator defined by $\Box = \X^{2} + \Y^{2} - \Theta^{2}$.  The Casimir operator is in the center of the universal enveloping algebra of $\Sl(2,\RR)$, and so it acts as a multiplicative scalar in each irreducible representation.  The value of this scalar classifies the irreducible representations of $\SL(2,\RR)$, so we will denote by $\HH_{\mu}$ the representation where $\Box$ acts by $-\mu$.

Any unitary representation of $H_{1}$ decomposes as a direct integral of $\HH_{\mu}$'s.  If there exists a $\mu_{0}$ such that $0 < \mu_{0} < \mu$ for all $\HH_{\mu}$ appearing in this decomposition, we say that the unitary representation has a \emph{spectral gap} for the Casimir operator.  

\subsubsection{Representations of $\h_{2} = \RR \X$} \label{h2}
Since $H_{2}$ is abelian, any element $\ex(t \X)$ acts as the multiplicative scalar $e^{it \theta}$ in an irreducible unitary representation $\HH_{\theta}$, for some real $\theta$.  For our purposes, the most important feature of $\HH_{\theta}$ is that it is one dimensional.  As such, we can pick a smooth vector $v_{\theta} \in \HH_{\theta}$ of norm $1$ as a basis.

\subsection{Invariant distributions and vanishing of obstructions} \label{obstructions}
In general, given an irreducible unitary representation of a Lie algebra $\h$ on a Hilbert space $\HH$, one has obstructions to solving the cohomology equation coming from distributions that are invariant under the flow.  For example, given $\U \in \h$ and $w \in C^{\infty}(\HH)$, in order to solve the equation $\U v = w$, one must have that $D(w)=0$ for every $\U$-invariant distribution $D$.  The set of $\U$-invariant distributions on a representation $\HH$ is denoted $\mathcal{I}_{\U}(\HH)$.  

Now, we recall our situation, where $\h = \h_{1} \times \h_{2}$, $\h_{1}=\Sl(2,\RR)$, and $\U$,$\X$ are in $\h_{1}$,$\h_{2}$, respectively.  (See Section \ref{definingh}.)  Consider an irreducible representation $\HH_{\mu,\theta}=\HH_{\mu} \otimes \HH_{\theta}$, and a cocycle given by $f_{\mu,\theta}$,$g_{\mu,\theta} \in C^{\infty}(\HH_{\mu,\theta})$ satisfying $\U g_{\mu,\theta} = \X f_{\mu,\theta}$.  We write $f_{\mu,\theta} = f_{\mu} \otimes v_{\theta}$ and $g_{\mu,\theta} = g_{\mu} \otimes v_{\theta}$ for some $f_{\mu}$ and $g_{\mu}$ in $\HH_{\mu}$, recalling that $v_{\theta} \in \HH_{\theta}$ is the norm $1$ basis discussed in Section \ref{h2}

The goal of this section is to show that the obstructions coming from the first factor vanish.  More precisely, we show that if $D$ is a $\U$-invariant distribution on $W^{s}(\HH_{\mu})$, the Sobolev space of order $s \geq 0$, then $D(f_{\mu})=0$.  (It will turn out that this is enough to write down a solution to the cohomology equation in $\HH_{\mu} \otimes \HH_{\theta}$.)

The following lemma was communicated to us by L. Flaminio in a more general form than the one in which we present it; we give a statement and proof that applies specifically to our setup.  Keeping the same notation as above, $\U \in \h_{1}$, $\X \in \h_{2}$, and $\HH_{\mu} \otimes \HH_{\theta}$ is an irreducible representation of $H$.  

\begin{lemma}[Flaminio]\label{flaminiolemma}
Let $D \in \II_{\U}(W^{s}(\HH_{\mu}))$, where $s \geq 0$.  Define $\bar{D}: W^{s}(\HH_{\mu}) \otimes \HH_{\theta} \rightarrow \HH_{\theta}$ by
\[
\bar{D} = D \otimes 1.
\]
That is, for all $u \in W^{s}(\HH_{\mu})$ and $v \in \HH_{\theta}$, $\bar{D}(u \otimes v) = D(u)v$.  Suppose that $f,g \in C^{\infty}(\HH_{\mu} \otimes \HH_{\theta})$ satisfy $\U g = \X f$.  Furthermore, suppose that the equation $\X w = 0$ implies that $w = 0$.  Then $\bar{D}(f) = 0$.
\end{lemma}
\begin{proof}
By the $\U$-invariance of $D$, we see that for any $u \otimes v \in W^{s}(\HH_{\mu}) \otimes \HH_{\theta}$, 
\[
\bar{D}(\U (u \otimes v)) = \bar{D}((\U u) \otimes v) = D(\U u)v = 0.
\]
Therefore, we have $\bar{D}(\U w) = 0$ for all $w \in W^{s}(\HH_{\mu}) \otimes \HH_{\theta}$.  Now, the diagram
\begin{displaymath}
\xymatrix{
  W^{s}(\HH_{\mu}) \otimes \HH_{\theta} \ar[r]^{\qquad \bar{D}} \ar[d]_{1\otimes\pi_{\theta}} & 
  																													\HH_{\theta} \ar[d]^{\pi_{\theta}} \\
  W^{s}(\HH_{\mu}) \otimes \HH_{\theta} \ar[r]^{\qquad \bar{D}} & \HH_{\theta} }  
\end{displaymath}
commutes, where $\pi_{\theta}$ denotes the representation of $\h_{2}$ on $\HH_{\theta}$, and the vertical arrows correspond to the map obtained by choosing some element of $\h_{2}$.  So, we have that 
\[
\X \bar{D}(f) = \bar{D}(\X f) = \bar{D}(\U g) = 0.
\]
By the last assumption in the Lemma, this implies that $\bar{D}(f) = 0$.
\end{proof}

In our situation we indeed have that the equation $\X w = 0$ implies $w=0$.  (This follows from ergodicity of the flow of $\X$ on $G/\G$.)  Therefore, we can apply Lemma \ref{flaminiolemma} to see that $\bar{D}(f_{\mu,\theta})=0$ for any $D \in \II_{\U}(W^{s}(\HH_{\mu}))$.  But,
\begin{eqnarray}
 \bar{D}(f_{\mu,\theta}) & = & \bar{D}(f_{\mu} \otimes v_{\theta}) \nonumber \\
 & = & D(f_{\mu})v_{\theta} \nonumber \\
 & = & 0, \nonumber
\end{eqnarray}
therefore,
\[
D(f_{\mu}) = 0.
\]
In the next section, we will use this to write down a solution $P_{\mu,\theta} \in \HH_{\mu} \otimes \HH_{\theta}$.

\subsection{Solutions in irreducible representations of $\h$}
In \cite{FF}, Flaminio and Forni proved the following theorem which shows that, for horocycle flows on quotients of $\PSL(2,\RR)$, the $\U$-invariant distributions are the only obstructions to solving the cohomology equation for a given function $f \in W^{s}(\HH_{\mu})$.  Furthermore, the solution comes with a fixed loss of regularity.  (A calculation on the second principal series representations shows that the result also holds for $\SL(2,\RR)/\G$.)  We will apply this to $f_{\mu}$ to get a solution on the first factor, and then use this to write a solution in $\HH_{\mu} \otimes \HH_{\theta}$.

\begin{theorem}[Flaminio-Forni]\label{FF}
Let $s>1$.  If $\mu > \mu_{0} > 0$, then there exists a constant $C_{\mu_{0},s,t}$ such that for all $f \in W^{s}(\HH_{\mu})$,
\begin{itemize}
  \item if $t < -1$, or
  \item if $t < s-1$ and $D(f)=0$ for all $D \in \mathcal{I}_{\U}(W^{s}(\HH_{\mu}))$,
\end{itemize}
then the equation $\U P = f$ has a solution $P \in W^{t}(\HH_{\mu})$, which satisfies the Sobolev estimate $\left\|P \right\|_{t} \leq C_{\mu_{0},s,t} \left\|f \right\|_{s}$.  Solutions are unique modulo the trivial subrepresentation if $t > 0$.
\end{theorem}

In Section \ref{obstructions}, we saw that for any $D \in \mathcal{I}_{\U}(W^{s}(\HH_{\mu}))$, $D(f_{\mu})=0$.  Therefore, we can apply Theorem \ref{FF} to obtain a solution $P_{\mu} \in \HH_{\mu}$ to the equation $\U P_{\mu}=f_{\mu}$, satisfying the estimate $\left\|P_{\mu} \right\| \leq C_{\mu_{0},1+\e,0} \left\|f _{\mu}\right\|_{1+\e}$, where $0 < \mu_{0} < \mu$ and $\e > 0$.  Set $P_{\mu,\theta} = P_{\mu} \otimes v_{\theta}$.

\begin{lemma} \label{irredsolution}
$P_{\mu,\theta}$ as defined above is a solution to $\U P_{\mu,\theta} = f_{\mu,\theta}$ in the irreducible representation $\HH_{\mu} \otimes \HH_{\theta}$, and it satisfies the estimate
\[
\left\|P_{\mu,\theta}\right\| \leq C_{\mu_{0},1+\e,0} \left\|f_{\mu,\theta} \right\|_{1+\e} 
\]
for any $\e > 0$.
\end{lemma}
\begin{proof}
The first assertion follows from 
\begin{eqnarray}
\U P_{\mu,\theta} & = & \U (P_{\mu} \otimes v_{\theta}) \nonumber \\
 & = & (\U P_{\mu}) \otimes v_{\theta} \nonumber \\
 & = & f_{\mu} \otimes v_{\theta} \nonumber \\
 & = & f_{\mu,\theta}. \nonumber
\end{eqnarray}
The norm estimate comes from combining
\[
\left\|P_{\mu}\right\| = \left\|P_{\mu,\theta} \right\| 
\]
and
\[
\left\|f _{\mu}\right\|_{1+\e} \leq \left\|f_{\mu,\theta} \right\|_{1+\e}
\]
with the estimate on $P_{\mu}$ obtained from applying Theorem \ref{FF}.
\end{proof}

\subsection{Global solution} \label{globalsolutionsection}

We are now prepared to build a global solution $P$ in any unitary representation $\HH$ of $H_{1} \times H_{2}$ that has a spectral gap for the Casimir operator from $H_{1}$.  That is, we can prove Lemma \ref{globalsolution}.

\begin{proof}[Proof of Lemma \ref{globalsolution}]
We have the decomposition
\[
\HH = \int_{\oplus}{\HH_{\mu} \otimes \HH_{\theta}ds(\mu,\theta)}
\]
where each irreducible $\HH_{\mu} \otimes \HH_{\theta}$ appears with some multiplicity $m(\mu,\theta)$.  We then decompose $f$ and $g$ as
\[
f = \int_{\oplus}{f_{\mu,\theta}ds(\mu,\theta)}, \quad f_{\mu,\theta} \in \HH_{\mu} \otimes \HH_{\theta}
\]
and
\[
g = \int_{\oplus}{g_{\mu,\theta}ds(\mu,\theta)}, \quad g_{\mu,\theta} \in \HH_{\mu} \otimes \HH_{\theta}.
\]
From Lemma \ref{irredsolution}, we have solutions $P_{\mu,\theta}$ in each irreducible representation $\HH_{\mu} \otimes \HH_{\theta}$.  Set
\[
P = \int_{\oplus}{P_{\mu,\theta}ds(\mu,\theta)}.
\]
Then it is clear that $\U P = f$, formally.

To see that $P \in \HH$, we use the estimates on the norms of the $P_{\mu, \theta}$.  We have
\begin{eqnarray}
\left\| P \right\|^{2} & = & \int_{\oplus}{\left\| P_{\mu,\theta} \right\|^{2} ds(\mu,\theta)} \nonumber \\
 & \leq & \int_{\oplus}{C_{\mu_{0},1+\e,0} \left\| f_{\mu,\theta} \right\|_{1+\e}^{2} ds(\mu,\theta)} \nonumber \\
 & = & C_{\mu_{0},1+\e,0} \left\| f \right\|_{1+\e}^{2} \nonumber
\end{eqnarray}
where $0 < \mu_{0} < \mu$ for all $\mu$ that appear in the decomposition of $\HH$, and $\e > 0$.  This proves that $P \in \HH$.
\end{proof}

The restriction of $L^{2}(G/\G)$ to $H_{1}$ has a spectral gap.  This follows from work of D. Kleinbock and G. Margulis in \cite{KM} which, when combined with a theorem of Y. Shalom in \cite{Shalom}, yields the following theorem, quoted from \cite{M1}.

\begin{theorem} \label{spectralgap}
Let $G = G_{1} \times \cdots \times G_{k}$ be a product of noncompact simple Lie groups, $\G \subset G$ an irreducible lattice, and $H \subset G$ a non-amenable closed subgroup.  Then the restriction of $L^{2}(G/\G)$ to $H$ has a spectral gap.
\end{theorem}

Therefore, we can apply Lemma \ref{globalsolution} to obtain a solution $P \in L^{2}(G/\G)$ to the equation $\U P = f$.  Now, for a fixed $t \in \RR$, we have
\begin{eqnarray}
\U (P(\phi_{t}^{\X}x) - P(x)) & = & \frac{d}{ds}[P(\phi_{t}^{\X}\phi_{s}^{\U}x) - P(\phi_{s}^{\U}x)]_{s=o} \nonumber \\
 & = & \frac{d}{ds}[P(\phi_{s}^{\U}\phi_{t}^{\X}x)]_{s=0} - \frac{d}{ds}[P(\phi_{s}^{\U}x)]_{s=o}, \nonumber
\end{eqnarray}
since the flows of $\X$ and $\U$ commute.  Then, because $\U P = f$, 
\begin{eqnarray} 
 & = & f(\phi_{t}^{\X}x) - f(x) \nonumber \\
 & = & \int_{0}^{t}{\frac{d}{d\t}[f(\phi_{\t}^{\X}x)]d\t} \nonumber \\
 & = & \int_{0}^{t}{\X f(\phi_{\t}^{\X}x)d\t}, \nonumber
\end{eqnarray}
which, by the identity $\U g = \X f$,
\begin{eqnarray} 
 & = & \int_{0}^{t}{\U g(\phi_{\t}^{\X}x)d\t} \nonumber \\
 & = & \U (\int_{0}^{t}{g(\phi_{\t}^{\X}x)d\t}). \nonumber
\end{eqnarray}
Since the flow of $\U$ on $G/\G$ is ergodic, this implies that 
\[
P(\phi_{t}^{\X}x) - P(x) = \int_{0}^{t}{g(\phi_{\t}^{\X}x)d\t}.
\]
One sees that the right hand side is differentiable in $t$.  Differentiating, we obtain $\X P = g$.  Thus, $P$ simultaneously solves $\U P = f$ and $\X P = g$.

Our task in the next section is to show that $P$ is smooth.

\subsection{Smoothness of global solution; proof of Theorem \ref{nilpotentsemisimple}} \label{smoothnessofP}

By the assumption on $\X \in \g$, we have a splitting of the tangent bundle of $G/\G$,
\[
T(G/\G) = E^{-} \oplus E^{0} \oplus E^{+},
\]
where $E^{-}$ and $E^{+}$ are the stable and unstable distributions with respect to the flow of $\X$; that is, there exist constants $A, B, \L_{-}, \L_{+} \in \RR_{+}$ such that
\[
\left\| d\phi_{t}^{\X}(\V) \right\| \leq A \cdot e^{-t \L_{-}} \cdot \left\| \V \right\|
\]
for all $\V \in E^{-}$ and $t >0$, and
\[
\left\| d\phi_{-t}^{\X}(\W) \right\| \leq B \cdot e^{-t \L_{+}} \cdot \left\| \W \right\|
\]
for all $\W \in E^{+}$ and $t > 0$.  Furthermore, we have assumed that the intersections $E^{-} \cap \g_{i}$ and $E^{+} \cap \g_{i}$ are nontrivial for all $i = 1, \ldots, k$.  The distributions $E^{-}$ and $E^{+}$ integrate to the stable and unstable foliations for the flow $\phi_{t}^{\X}$ on $G/\G$, denoted $W^{-}$ and $W^{+}$, respectively.  For $y \in W^{-}(x)$ and $z \in W^{+}(x)$, we have
\[
\dist(\ex(t \X)x, \ex(t \X)y) \leq A \cdot e^{-t \L_{-}} \cdot \dist(x,y),
\]
and
\[
\dist(\ex(-t \X)x, \ex(-t \X)z) \leq B \cdot e^{-t \L_{+}} \cdot \dist(x,z),
\]
for all $t > 0$.

We will begin our proof that the solution $P \in L^{2}(G/\G)$ is smooth by examining how $P$ behaves along leaves of the foliations $W^{-}$ and $W^{+}$.  The following lemma will establish that $P$ satisfies a Lipschitz continuity condition locally on these leaves.  

\begin{lemma} \label{lipschitz}
For almost every $x \in G/\G$, there is a neighborhood $V_{x} \subset W^{-}(x)$ containing $x$ such that for almost every $y \in V_{x}$, the following holds:
\[
\left| P(x) - P(y) \right| \leq K_{-} \cdot \mathrm{dist}(x,y),
\]
where $K_{-} > 0$ is a constant.  Similarly, there is a neighborhood $V'_{x} \subset W^{+}(x)$ such that for almost every $y \in V'_{x}$, the following holds:
\[
\left| P(x) - P(y) \right| \leq K_{+} \cdot \mathrm{dist}(x,y),
\]
where $K_{+} > 0$ is a constant.
\end{lemma}

\begin{proof}
To begin, note that for any $x, y \in G/\G$,
\begin{eqnarray}
\left| P(y) - P(x) \right| & = & \nonumber \\
 &    & |P(y) - P(\ex(t \X) y) \\
 & + & P(\ex(t \X) y) - P(\ex(t \X)x) \\
 & + & P(\ex(t \X)x) - P(x)|. 
\end{eqnarray}
Combining lines $(3)$ and $(5)$, we have
\begin{eqnarray}
\left| P(y) - P(x) \right| & = & |\int_{0}^{t}{(\X P(\ex(\t \X)x) - \X P(\ex(\t \X) y))d\t} \nonumber \\
 & + & P(\ex(t \X) y) - P(\ex(t \X)x)|. \nonumber \\
 & = & |\int_{0}^{t}{(g(\ex(\t \X)x) - g(\ex(\t \X) y))d\t} \nonumber \\
 & + & P(\ex(t \X) y) - P(\ex(t \X)x)|. \nonumber
\end{eqnarray}
We will show that for almost every $x \in G/\G$ and almost every $y$ in some neighborhood $V_{x} \subset W^{-}(x)$ containing $x$, there is an increasing divergent sequence $\{t_{k}\}$ such that 
\[
|P(\ex(t_{k} \X) y) - P(\ex(t_{k} \X) x)| \longrightarrow 0.
\]

We begin by noting that, since $\X P = g$ is smooth and $G/\G$ is compact, $g$ is Lipschitz continuous on $G/\G$.  That is, for all $x,y \in G/\G$, we have
\[
\left| g(x) - g(y) \right| \leq C \cdot \mathrm{dist}(x,y)
\]
for some $C > 0$.

We cover $G/\G$ by a collection of coordinate charts of the form $U \times V$, where $\{ z \} \times V$ is a neighborhood of a stable leaf of $W^{-}$ for every $z \in U$.  Since the foliation is absolutely continuous, this can be done in such a way that Fubini's theorem holds in each of these charts, with respect to Lebesgue	 measures on $U$ and $V$.

Let $E \subset G/\G$ be a Luzin set for $P$ of measure $0.99$.  Then for almost every $x \in G/\G$,
\[
\frac{1}{T} \int_{0}^{T}{\chi_{E}(\ex(t\X)x)dt} \longrightarrow 0.99,
\]
as $T \rightarrow \infty$, where $\chi_{E}$ is the characteristic function for $E$.  Suppose $U_{x} \times V_{x}$ is a coordinate chart containing $x$.  By Fubini's Theorem, we also have that for almost every $x \in G/\G$, and almost every $y \in \{ p_{1}(x) \} \times V_{x}$, 
\[
\frac{1}{T} \int_{0}^{T}{\chi_{E}(\ex(t\X) y) dt} \longrightarrow 0.99.
\]
(Here, $p_{1}: U_{x} \times V_{x} \rightarrow U_{x}$ is projection onto the first coordinate.)  For such $x$ and $y$, there is an increasing divergent sequence $\{t_{k}\} \subset \RR_{+}$ such that $\ex(t_{k}\X)x$ and $\ex(t_{k}\X) y$ are in the Luzin set $E$ for all $k$.  Thus, for almost every $x \in G/\G$ and almost every $y \in \{ p_{1}(x) \} \times V_{x}$,
\[
|P(\ex(t_{k} \X) y) - P(\ex(t_{k} \X)x)| \longrightarrow 0.
\]

Now, for these $x \in G/\G$ and $y \in \{ p_{1}(x) \} \times V_{x}$,
\begin{eqnarray}
\left| P(y) - P(x) \right| & = & |\int_{0}^{\infty}{(g(\ex(\t \X)x) - g(\ex(\t \X)y))d\t}| \nonumber \\
 & \leq & \int_{0}^{\infty}{|(g(\ex(\t \X)x) - g(\ex(\t \X)y))|d\t} \nonumber \\
 & \leq & \int_{0}^{\infty}{C \cdot \mathrm{dist}(\ex(\t \X)y, \ex(\t \X)x)d\t} \nonumber \\
 & \leq & \int_{0}^{\infty}{C \cdot A \cdot \mathrm{dist}(y, x) \cdot e^{-\t \L_{-}} d\t} \nonumber \\
 & = & \frac{C \cdot A}{\L_{-}} \cdot \mathrm{dist}(x,y). \nonumber
\end{eqnarray}
This is the desired local Lipschitz condition along stable leaves for the flow of $\X$, with $K_{-} = \frac{C \cdot A}{\L_{-}}$. 

The preceding argument holds \emph{mutatis mutandis} for the unstable foliation, $W^{+}$.
\end{proof}

We use this Lipschitz condition in the following lemma, which establishes that $P$ can be differentiated in stable and unstable directions.

\begin{lemma} \label{livsic}
Suppose $P \in L^{2}(G/\G)$ satisfies $\X P = g$, where $\X \in \g$ is semisimple and $g \in C^{\infty}(L^{2}(G/\G))$.  Let $\V$ be a stable or unstable vector for the flow of $\X$.  Then $\V^{k} P \in L^{2}(G/\G)$ for all $k \in \NN$.
\end{lemma}
\begin{proof}
Without loss of generality, we assume $\V$ is a stable unit vector for $\X$; that is, $\V \in E^{-}$ and $\left\| \V \right\| = 1$.  The following argument can be carried out for unstable vectors by considering negative time.  

We now compute
\begin{eqnarray}
\V P(x) & = & \lim_{s \rightarrow 0}{\frac{P(\mathrm{exp}(s \V) x) - P(x)}{s}} \nonumber \\
 & = & \lim_{s \rightarrow 0}{\frac{1}{s}(P(\mathrm{exp}(s \V) x) - P(\mathrm{exp}(t \X) \mathrm{exp}(s \V) x))}  \\
 & + & \lim_{s \rightarrow 0}{\frac{1}{s}(P(\mathrm{exp}(t \X) \mathrm{exp}(s \V) x) - P(\mathrm{exp}(t \X) x))}  \\
 & + & \lim_{s \rightarrow 0}{\frac{1}{s}(P(\mathrm{exp}(t \X) x) - P(x))} 
\end{eqnarray}
where $t \in \RR_{+}$.  Combining lines $(6)$ and $(8)$, we have
\begin{eqnarray}
\V P(x) & = & \lim_{s \rightarrow 0}{\frac{1}{s}\int_{0}^{t}{(\X P(\ex(\t \X)x) - \X P(\ex(\t \X)\ex(s \V)x))d\t}} \nonumber \\
 & + & \lim_{s \rightarrow 0}{\frac{1}{s}(P(\ex(t \X)\ex(s \V)x) - P(\ex(t \X)x))}. \nonumber
\end{eqnarray}
Setting $g_{\t}(x) := g(\ex(\t \X) x)$,
\begin{eqnarray}
\V P(x) & = & \lim_{s \rightarrow 0}{\frac{1}{s}\int_{0}^{t}{(g_{\t}(x) - g_{\t}(\ex(s \V)x))d\t}} \nonumber \\
 & + & \lim_{s \rightarrow 0}{\frac{1}{s}(P(\ex(t \X)\ex(s \V)x) - P(\ex(t \X)x))} \nonumber \\
 & = & -\lim_{s \rightarrow 0}{\frac{1}{s}\int_{0}^{t}{\int_{0}^{s}{\V g_{\t}(\ex(\s \V) x)d\s d\t}}} \nonumber \\
 & + & \lim_{s \rightarrow 0}{\frac{1}{s}(P(\ex(t \X)\ex(s \V)x) - P(\ex(t \X)x))} \nonumber \\
 & = & -\int_{0}^{t}{\lim_{s \rightarrow 0}{\frac{1}{s}\int_{0}^{s}{\V g_{\t}(\ex(\s \V) x)d\s d\t}}} \nonumber \\
 & + & \lim_{s \rightarrow 0}{\frac{1}{s}(P(\ex(t \X)\ex(s \V)x) - P(\ex(t \X)x))} \nonumber \\ 
 & = & -\int_{0}^{t}{\V g_{\t}(x)d\t}  \nonumber \\
 & + & \lim_{s \rightarrow 0}{\frac{1}{s}(P(\ex(t \X)\ex(s \V)x) - P(\ex(t \X)x))}. \nonumber
\end{eqnarray}
Since this expression is constant in $t$, we can take a limit,
\begin{eqnarray}
\V P(x) & = & - \lim_{t \rightarrow \infty}{\int_{0}^{t}{\V g_{\t}(x)d\t}}  \nonumber \\
 & + & \lim_{t \rightarrow \infty}{\lim_{s \rightarrow 0}{\frac{1}{s}(P(\ex(t \X)\ex(s \V)x) - P(\ex(t \X)x))}}.
\end{eqnarray}
By Lemma \ref{lipschitz}, we have control over line $(9)$ for almost every $x$ in the following way:
\begin{eqnarray}
 & \left| \lim_{t \rightarrow \infty}{\lim_{s \rightarrow 0}{\frac{1}{s}(P(\ex(t \X)\ex(s \V)x) - P(\ex(t \X)x))}} \right| \nonumber \\
\leq & \lim_{t \rightarrow \infty}{\lim_{s \rightarrow 0}{\frac{K_{-} \cdot A \cdot e^{-t \L_{-}}}{s} \cdot \dist(\ex(s \V)x,x) }} \nonumber \\
\leq & \lim_{t \rightarrow \infty}{\lim_{s \rightarrow 0}{\frac{K_{-} \cdot A \cdot e^{-t \L_{-}}}{s} \cdot s}} \nonumber \\
= & 0. \nonumber
\end{eqnarray}
So we are left with
\begin{eqnarray}
\V P(x) = -\int_{0}^{\infty}{\V g_{\t}(x)d\t}.
\end{eqnarray}

The following calculations will show that $(10)$ defines an $L^{2}$-function on $G/\G$.  Since $\V \in E^{-}$,
\begin{eqnarray}
\left| \int_{0}^{t}{\V g_{\t}(x)d\t} \right| & \leq & \int_{0}^{t}{A \cdot e^{-\t \L_{-}} \cdot \left| \V g(\ex(\t \X)x) \right| d\t}. \nonumber
\end{eqnarray}
We define the functions
\[
h_{t}(x) = \int_{0}^{t}{A \cdot e^{-\t \L_{-}} \cdot \left| \V g(\ex(\t \X)x) \right| d\t}
\]
and
\[
H_{t}(x) = -\int_{0}^{t}{\V g_{\t}(x)d\t} 
\]
for $t \in \RR_{+}$.  Then we have that $\left| H_{n}(x) \right| \leq h_{n}(x)$ for all $n \in \NN$.  Denoting Haar measure on $G/\G$ by $\mu$, we have 
\begin{eqnarray}
\left\| h_{t} \right\|_{L^{2}}^{2} & = & \int_{G/\G}{\left| \int_{0}^{t}{A \cdot e^{-\t \L_{-}} \cdot \V g(\ex(\t \X)x)d\t} \right|^{2}d\mu} \nonumber \\
 & \leq & \int_{G/\G}{\int_{0}^{t}{\left| A \cdot e^{-\t \L_{-}} \cdot \V g(\ex(\t \X)x) \right|^{2} d\t}d\mu} \nonumber \\
 & = & \int_{0}^{t}{\int_{G/\G}{A^{2} \cdot e^{-2\t \L_{-}} \cdot \left| \V g(\ex(\t \X)x) \right|^{2} d\mu}d\t} \nonumber \\
 & = & \int_{0}^{t}{A^{2} \cdot e^{-2\t \L_{-}} \cdot \left\| \V g \right\|_{L^{2}}^{2} d\t}. \nonumber
\end{eqnarray}
It is easy to see that the sequence $\left\{ h_{n} \right\}_{n \in \NN} \subset L^{2}(G/\G)$ is Cauchy, so converges in $L^{2}(G/\G)$.  Now, the sequence $\{H_{n}\}$ is dominated by $\{h_{n}\}$, therefore, by the Dominated Convergence Theorem, $\V P \in L^{2}(G/\G)$.

We now show that $\V^{2}P(x) \in L^{2}(G/\G)$.  It will be apparent that one can apply $\V$ successively with the same procedure.  First, we apply $\V$ to expression $(10)$ to yield
\begin{eqnarray}
\left| \V^{2}P(x) \right| & = & \left| -\lim_{s \rightarrow 0}{\frac{1}{s} \int_{0}^{\infty}{(\V g_{\t}(\ex(s \V)x) - \V g_{\t}(x))d\t}} \right| \nonumber \\
 & \leq & \lim_{s \rightarrow 0}{\int_{0}^{\infty}{\frac{1}{s} \left| \V g_{\t}(\ex(s \V)x) - \V g_{\t}(x) \right| d\t}} \nonumber \\
 & \leq & \lim_{s \rightarrow 0}{\int_{0}^{\infty}{\frac{1}{s} \cdot A \cdot e^{-\t \L_{-}} \left| \V g(\ex(s \V)\ex(\t \X)x) - \V g(\ex(\t \X)x) \right| d\t}} \nonumber
\end{eqnarray}
Since $\V g$ is smooth on $G/\G$, we have that 
\[
\frac{1}{s} \left| \V g(\ex(s \V)\ex(\t \X)x) - \V g(\ex(\t \X)x) \right| \leq M
\]
for all $s > 0$, and some $M > 0$.  Therefore, the integrand is dominated by $M(\t) = M \cdot A \cdot e^{-\t \L_{-}}$.  Thus, by the Dominated Convergence Theorem, we can bring the limit inside to see that $\V^{2}P \in L^{2}(G/\G)$.  Furthermore, one can repeat this procedure, applying $\V$ to $(10)$, to see that $\V^{k} P \in L^{2}(G/\G)$ for all $k$.  
\end{proof}

We will use the following lemma to show that the stable and unstable directions span $\g$ as a Lie algebra, that is, by taking successive brackets.  By Theorem \ref{KaSp}, this will imply that $P$ is smooth on $G/\G$.

\begin{lemma} \label{spanning}
Suppose $\g$ is a simple Lie algebra, and $\X \in \g$ is a semisimple element with nonzero stable and unstable vectors in $\g$.  Consider the splitting
\[
\g = E^{-} \oplus E^{0} \oplus E^{+}
\]
into stable and unstable directions.  Let $\LL \subset \g$ be the subalgebra generated by $E^{-}$ and $E^{+}$.  Then $\LL = \g$.
\end{lemma}
\begin{proof}
We will show that $\LL \subset \g$ is an ideal.  Note that every element of $\LL$ is a sum of elements of the form
\[
\V = [\V_{1},[\V_{2},[\V_{3}, \cdots,[\V_{k-1}, \V_{k}] \cdots ]]]
\]
where $\V_{i}$ is either in $E^{-}$ or $E^{+}$.  Suppose $\W \in E^{0}$.  By repeatedly applying the Jacobi identity, we can express $[\V, \W]$ as a sum of terms of the form
\[
\W_{\s} = \pm [\V_{\s(1)},[\V_{\s(2)},[\V_{\s(3)}, \cdots,[\V_{\s(k)}, \W] \cdots ]]]
\]
where $\s$ is a permutation on the set $\left\{ 1, 2, 3, \ldots, k \right\}$.  It is easy to see that if $\V_{\s(k)}$ is stable, then so is $[\V_{\s(k)}, \W]$; similarly, if $\V_{\s(k)}$ is unstable, then so is $[\V_{\s(k)}, \W]$.  Therefore, $\W_{\s} \in \LL$ and $[\V, \W] \in \LL$.  This proves that $\LL$ is an ideal in $\g$.  $\LL$ contains nonzero elements, therefore, $\LL = \g$.
\end{proof}

We are now ready to state the proof of the first main theorem.
\begin{proof}[Proof of Theorem \ref{nilpotentsemisimple}]
We have a semisimple Lie group $G$ with finite center, $\G \subset G$ a lattice, $\U \in \g$ nilpotent , and $\X \in \g$ semisimple and commuting with $\U$, such that the flow $\phi_{t}^{\X}$ has stable and unstable directions in the Lie algebra of each factor of $G$.  We have $f, g \in C^{\infty}(L^{2}(G/\G))$ satisfying $\U g = \X f$, and $\int_{G/\G}{f} = \int_{G/\G}{g} = 0$.

By the Jacobson-Morozov Lemma (Theorem \ref{JM}), we can find the subalgebra $\h:= \Sl(2,\RR) \times \RR \X \subset \g$ such that $\U =  \big( \begin{smallmatrix} 0 & 1 \\ 0 & 0 \end{smallmatrix}\big) \times (0) \in \Sl(2,\RR) \times \RR \X$. The corresponding subgroup of $\h$ is $H=H_{1} \times H_{2} \subset G$. 

The left-regular unitary representation of $H$ on $L^{2}(G/\G)$ decomposes as
\[
L^{2}(G/\G) = \int_{\oplus}{\HH_{\mu} \otimes \HH_{\theta}ds(\mu,\theta)},
\]
where $ds$-almost every $\HH_{\mu} \times \HH_{\theta}$ is irreducible, so we restrict our attention to an irreducible $\HH_{\mu} \otimes \HH_{\theta}$.  By Lemma \ref{flaminiolemma}, the obstructions to solving $\U P = f_{\mu,\theta}$ coming from $\U$-invariant distributions vanish in each irreducible $\HH_{\mu} \otimes \HH_{\theta}$.  With this, we apply Theorem \ref{FF} to find a solution $P_{\mu} \in \HH_{\mu}$.  By Lemma \ref{irredsolution}, $P_{\mu,\theta} = P_{\mu} \otimes v_{\theta}$  is a solution to $\U P_{\mu,\theta} = f_{\mu,\theta}$ in $\HH_{\mu} \otimes \HH_{\theta}$, and it satisfies the estimate
\[
\left\|P_{\mu,\theta}\right\| \leq C_{\mu_{0},1+\e,0} \left\|f_{\mu,\theta} \right\|_{1+\e},
\]
where $0 < \mu_{0} < \mu$.

Now, Theorem \ref{spectralgap} guarantees that the regular representation of $H$ on $L^{2}(G/\G)$ has a spectral gap for the Casimir operator from $H_{1}$.  Therefore, by Lemma \ref{globalsolution} we can glue the $P_{\mu,\theta}$'s together to get a solution $P \in L^{2}(G/\G)$ to the equation $\U P = f$.  By ergodicity of the flow of $\U$ on $G/\G$, we also get that $\X P = g$ (see the discussion at the end of Section \ref{globalsolutionsection}).

By Lemma \ref{livsic}, $\V^{k}P \in L^{2}(G/\G)$ for any $\V \in \g$ that is stable or unstable with respect to $\X$.   By assumption on $\X$, for each $i = 1, \ldots, k$, we have the decomposition
\[
\g_{i} = E_{i}^{-} \oplus E_{i}^{0} \oplus E_{i}^{+}
\]
into stable and unstable directions for the flow $\phi_{t}^{\X}$.  By Lemma \ref{spanning}, these directions span each $\g_{i}$ as a Lie algebra.  Therefore the distributions $E^{-}$ and $E^{+}$ span $\g$ as a Lie algebra, so we can apply Theorem \ref{KaSp} to see that $P$ is smooth.  This proves the theorem.
\end{proof}

\section{Proof of Theorem \ref{nilpotentnilpotent}; proof of Theorem \ref{maximalunipotent}}

\subsection{Strategy}
Let $H = \overline{\SL(2,\RR)}^{k} \times \overline{\SL(2,\RR)}^{l}$ be the product of two finite-sheeted covers of $\SL(2,\RR)$, and let $U \in H$ be the unipotent subgroup obtained by exponentiating $\U_{1} =  \big( \begin{smallmatrix} 0 & 1 \\ 0 & 0 \end{smallmatrix}\big) \times (0)$ and $\U_{2} = (0) \times \big( \begin{smallmatrix} 0 & 1 \\ 0 & 0 \end{smallmatrix}\big) \in \Sl(2,\RR) \times \Sl(2,\RR)$. Given an embedding $i:H \hookrightarrow G$ into a noncompact semisimple Lie group with finite center, and a smooth cocycle $\a$ over the $U$-action on $G/\G$, Mieczkowski's results imply a solution $P \in L^{2}(G/\G)$ to the cohomology e`quation that is smooth in directions tangent to the $H$-orbits in $G$.  Our ultimate goal is to show that $P$ is actually smooth in all directions.

Suppose $i':H \hookrightarrow G$ is a different embedding, and that $i|_{U} = i'|_{U}$.  Then there is another transfer function $Q \in L^{2}(G/\G)$ that is smooth in directions tangent to the $H$-orbits corresponding to this new embedding.  An ergodicity argument will show that $P$ and $Q$ differ by a constant, which can be chosen to be zero.  Finally, we will show that there are enough embeddings of $H$ into $G$ that coincide on $U$ to prove that $P$ is smooth in all directions.

The following sections are devoted to proving these assertions.

\subsection{Obtaining transfer functions}
In this section we show that the results in \cite{M1} and \cite{M2} can be applied to show that there are transfer functions that are smooth in the $H$-orbit directions of $G$.

Let $\a$ be a smooth cocycle over the action of $U$ on $G/\G$.  Its infinitesimal generator $\o$ is completely determined by where it sends the generators $\U_{1}$ and $\U_{2}$ of $\uu$.  In other words, it is determined by the functions
\[ f = \o(\U_{1}) \qquad g = \o(\U_{2}). \]
Now the cocycle identity is
\[ \U_{1}g = \U_{2}f \]
and the cohomology equation is
\[ \U_{1}P = f \quad \textrm{and} \quad \U_{2}P = g. \]

Suppose we have a unitary representation of $\SL(2,\RR) \times \SL(2,\RR)$ on the Hilbert space $\HH$.  Mieczkowski shows that if the Casimir element for both factors has a spectral gap, then there is a smooth vector $P \in C^{\infty}(\HH)$ that is a solution to the cohomology equation.  In fact, he proves the following stronger result.
\begin{theorem}[Mieczkowski]\label{Mi1}
If there exists a $\mu_{0} > 0$ such that the spectrum of each Casimir satisfies $\s(\Box_{i}) \cap (0,\mu_{0}) = \emptyset$, then we have the following.  Let $f,g \in W^{2s}(\HH)$, ($s>1$), and satisfy the equation $\U_{2}f = \U_{1}g$.  If $t < s-1$, then there exist solutions $P,P' \in W^{t}(\HH)$ such that $\U_{1}P = f$ and $\U_{2}P' = g$.  Furthermore, the norms of $P,P'$ must satisfy $\lVert P \rVert_{t} \leq C_{\mu_{0},s,t} \lVert f \rVert_{2s}$, and $\lVert P' \rVert_{t} \leq C_{\mu_{0},s,t} \lVert g \rVert_{2s}$.  If $t > 1$, then $P$ and $P'$ must coincide, so that there is a true simultaneous solution.
\end{theorem}

Since the unitary representations of a finite sheeted cover of $\SL(2,\RR)$ are unitarily equivalent to those for $\SL(2,\RR)$, Theorem \ref{Mi1} holds for representations of $H = \overline{\SL(2,\RR)}^{k} \times \overline{\SL(2,\RR)}^{l}$.

An embedding $H \hookrightarrow G$ induces a unitary representation of $\overline{\SL(2,\RR)}^{k} \times \overline{\SL(2,\RR)}^{l}$ on $L^2(G/\G)$.  In order to apply the previous theorem, we need to show that the Casimir elements for both factors have spectral gaps.  But this is immediate from Theorem \ref{spectralgap}.  Therefore, we can apply Theorem \ref{Mi1}.  Our smooth cocycle $\a$ is determined by the smooth functions $f,g \in C^{\infty}(L^{2}(G/\G))$, and Theorem \ref{Mi1} guarantees the existence of the transfer function $P \in C^{\infty}(L^{2}(G/\G))$.

\subsection{Different embeddings}
We point out that if there are two different embeddings $i:H \hookrightarrow G$ and $i':H \hookrightarrow G$ that coincide on $U \subset H$, then the corresponding transfer functions $P$ and $Q$ differ by a constant.  This is a simple consequence of the ergodicity of the flow of $\U$ on $G/\G$.   We can choose the constant to be $0$, so the transfer functions $P$ and $Q$ that we get from the embeddings $i$ and $i'$ agree almost everywhere.  Furthermore, they are smooth along their respective $H$-orbits.  Therefore, the partial derivatives of $P$ in directions tangent to the $i'(H)$-orbits also exist, as $L^{2}$ functions.  Our next goal is to show that there are enough embeddings of $H$ into $G$ to span all directions with the orbits.

\subsection{Getting enough embeddings}
In this section it will be convenient to denote $H$ as being a subgroup, $H \subset G$.  Different embeddings that coincide on $U$ will be achieved by conjugating $H$ by elements of the centralizer $Z(U)$ of $U$ in $G$.  We will look at the images of the Lie algebra $\h$ under these conjugations and show that the Lie algebra generated by the union of these is all of $\g$, the Lie algebra of $G$.  Theorem \ref{KaSp} will then imply that the solution $P$ is smooth.

\begin{proposition}\label{prop}
Suppose $H$ is a finite-dimensional split semisimple Lie group, $U \subset H$ is a unipotent subgroup, and $G$ is a simple Lie group into which $H$ embeds.  Let $\uu$, $\h$, and $\g$ be their respective Lie algebras.  Denote by $Z(U)$ the centralizer of $U$ in $G$.  Let $\LL = \left\langle \Ad(Z(U)) \h \right\rangle$ be the Lie algebra generated by 
\[\Ad(Z(U)) \h = \left\{g\X g^{-1} | g \in Z(U) \quad \mathrm{ and } \quad \X \in \h \right\}.\]
Then $\LL = \g$.
\end{proposition}
\begin{proof}
The centralizer of $\uu$ in $\g$, denoted $\z(\uu)$, is the Lie algebra of $Z(U)$.  Notice that for all $\X \in \z(\uu)$ and $\Y \in \h$, 
\[ t \mapsto \exp({t \X}) \cdot \Y \cdot \exp{(-t \X)}\]
is a curve in $\LL$ with velocity $[\X, \Y]$ at $t=0$.  Therefore, $\left[\z(\uu),\h \right] \subset \LL$.

Since $\h$ is split, there is a splitting Cartan subalgebra $\mathfrak{t} \subset \h$ that acts diagonally on $\g$, and we can order its roots so that $\uu$ is spanned by the positive root spaces.  Then we have a decomposition of $\g$ into the sum
\[ 
\g = \bigoplus_{\l \in \Psi} \g^{\l} 
\]
where $\g^{\l}$ is the sum of all $\ad(\h)$-invariant subspaces of $\g$ with highest weight $\l$, and $\Psi$ is a finite set of highest weights.
For $\l \in \Psi$,
\[ 
\g_{\l} = \left\{\X \in \g^{\l} | [ \T, \X]=\l(\T)\X \quad \textrm{for all } \T \in \mathfrak{t}\right\}. 
\]

Since $\uu$ is in the positive root spaces, any element of $\uu$ annihilates any highest weight vector, so $\g_{\l} \subset \z(\uu)$ for all $\l \in \Psi$.  Now, for any $\X \in \g_{\l}$ and $\T \in \mathfrak{t}$, we have that $[\X,\T] \in [\z(\uu), \h] \subset \LL$.  But $[\X,\T]=-\l(\T)\X$, so if $\l \neq 0$, then $\X \in \LL$.  This shows that for $\l \neq 0$, $\g_{\l} \subset \LL$.  Since for any $\l \in \Psi$, $\g_{\l}$ generates $\g^{\l}$ as an $\h$-module, 
\[ \bigoplus_{\l \in \Psi \backslash \{0\}}\g^{\l} \subset \LL.\] 

Let $\mathfrak{i}$ be the Lie algebra generated by $\bigoplus_{\l \in \Psi \backslash \{0\}}\g^{\l}$.  Then it is clear that $\mathfrak{i} \subset \LL$, and that
\[ \g = \mathfrak{i} + \g^{0} = \mathfrak{i} + \z(\mathfrak{t}). \]

We claim that $\mathfrak{i}$ is $\ad(\z(\mathfrak{t}))$-invariant.  Let $\X$ be a non-zero (not necessarily highest) weight vector with weight $\l$, and let $\Z \in \z(\mathfrak{t})$.  Then, for any $\T \in \mathfrak{t}$,
\begin{displaymath}
[[\X,\Z],\T] = [[\X,\T],\Z]
 = \l(\T)[\X,\Z].
\end{displaymath}
Thus, $[\X,\Z]$ is a weight vector with weight $\l$.  This shows that the non-zero weight spaces are $\ad(\z(\mathfrak{t}))$-invariant, and since $\mathfrak{i}$ is the Lie algebra generated by these, it is also $\ad(\z(\mathfrak{t}))$-invariant.

Obviously, $\mathfrak{i}$ is also $\ad(\mathfrak{i})$-invariant, hence it is an ideal in $\g$.  Since $\g$ is simple, and $\mathfrak{i}$ contains more than just $0$, we see that $\mathfrak{i}$ must equal $\g$.  Finally, since $\mathfrak{i} \subset \LL$, we get the desired result that $\LL = \g$.
\end{proof}

Our $H$ is split semisimple.  We will use this lemma to show that there are enough conjugates of $\h$ in $\g$ by elements of $Z(U)$ to generate $\g$ as a Lie algebra.  This is all that is needed to prove Theorem \ref{nilpotentnilpotent}; the proof will be stated in the following section.

\subsection{Proofs of Theorem \ref{nilpotentnilpotent} and Theorem \ref{maximalunipotent}}

Here we present the proof of Theorem \ref{nilpotentnilpotent} and Theorem \ref{maximalunipotent}.  We will keep the same notation for $H$ and $U$ throughout.
\begin{proof}[Proof of Theorem \ref{nilpotentnilpotent}]
We have a product $G = G_{1} \times \cdots \times G_{2}$ of noncompact simple Lie groups with finite center.  We have assumed that $G$ admits an embedding of $H$ such that the projection $U_{i}$ of $U$ to $G_{i}$ is nontrivial for all $i = 1, \ldots, k$.  Suppose we are given a smooth cocycle $\a: U \times G/\G \rightarrow \RR$.

By the discussion following Theorem \ref{Mi1} \cite{M1}, there exists a transfer function $P \in L^{2}(G/\G)$ for the given smooth cocycle $\a$, and $P$ is smooth in directions tangent to the $H$-orbits corresponding to the given embedding $i: H \hookrightarrow G$.  We obtain different embeddings of $H$ into $G$ by conjugating the image of $i$ by elements of the centralizer $Z(U)$ of $U$ in $G$.  Such embeddings will clearly all agree on $U$.  $P$ is differentiable, in the $L^{2}$ sense, in directions that are tangent to the $H$-orbits corresponding to any of these embeddings.  

To see that there are enough such embeddings to span $\g$ as a Lie algebra, observe that the projection $H_{i}$ of $H$ to $G_{i}$ is a split semisimple Lie subgroup of $G_{i}$, for all $i$.  Proposition \ref{prop} then shows that there are enough conjugates of $\h_{i} := \mathrm{Lie}(H_{i})$ by elements of $Z_{G_{i}}(U_{i}) \subset Z(U)$ to span $\g_{i} := \mathrm{Lie}(G_{i})$.  Thus, there are enough conjugates of $\h$ by elements of $Z(U)$ to span $\g$.  Therefore, by Theorem \ref{KaSp}, $P$  is smooth on $G/\G$.  This completes the proof.
\end{proof}

\begin{proof}[Proof of Theorem \ref{maximalunipotent}]
Let $\a$ be a cocycle over the $V$-action on $G/\G$.  Then it restricts to a cocycle over the $U$-action on $G/\G$, so by the previous theorem there is a smooth transfer function $P$ that satisfies
\[ \a(u,x) = -P(ux) + c(u) + P(x) \]
for all $u \in U$ and $x \in G/\G$, where $c: U \rightarrow \RR$ is a constant cocycle.  Let $V'$ be the center of $V$.  Then for $v \in V'$,
\begin{eqnarray}
\a(v,x) & = & \a(uvu^{-1},x) \nonumber \\
 & = & \a(u^{-1},x) + \a(v,u^{-1}x) + \a(u,vu^{-1}x) \nonumber \\
 & = & -P(u^{-1}x) + c(u^{-1}) + P(x) \nonumber \\
 & & -P(vx) + c(u) + P(vu^{-1}x) \nonumber \\
 & & + \a(v,u^{-1}x) \nonumber \\
 & = & -P(vx) + P(x) \nonumber \\
 & & -P(u^{-1}x) + P(vu^{-1}x) + \a(v,u^{-1}x) \nonumber
\end{eqnarray}
Regrouping terms, we see that 
\[ \a(v,x) + P(vx) - P(x) = -P(u^{-1}x) + P(vu^{-1}x) + \a(v,u^{-1}x) \]
 is a $U$-invariant smooth function on $G/\G$ for every $v \in V'$.  By ergodicity of the $U$-action on $G/\G$, it is constant.  Therefore, setting $c'(v) = -P(u^{-1}x) + P(vu^{-1}x) + \a(v,u^{-1}x)$, we have shown that $P$ satisfies
\[ \a(v,x) = -P(vx) + c'(v) + P(x) \]
for all $v \in V'$ and $x \in G/\G$.  It is clear that $c' = c$ on $U \cap V'$.

Now, $V'$ is closed and noncompact in $G$ and hence, by Theorem \ref{HoweMoore}, acts ergodically on $G/\G$.  Therefore, we can carry out the same calculation as above, where $V'$ will play the role that $U$ played, and $V$ will play the role that $V'$ played.  This shows that $P$ satisfies
\[ \a(v,x) = -P(vx) + c(v) + P(x) \]
for all $v \in V$ and $x \in G/\G$, and completes the proof of the theorem.
\end{proof}

\subsection{Remarks on the simple case}
Theorem A' is the statement of Theorem \ref{nilpotentsemisimple} for the case of a noncompact simple Lie group $G$ with finite center.  Notice that if $\ad(\X)$ has a root $\l$ that is not purely imaginary, then $\V + \overline{\V}$ is in the real Lie algebra $\g$, for any $\V \in \g \otimes \CC$ satisfying $[\X, \V] = \l \V$.  The vector $\V + \overline{\V}$  is either stable or unstable with respect to the flow $\phi_{t}^{\X}$ on $G/\G$, depending on whether the real part of $\l$ is negative or positive.  Thus, assumption in Theorem \ref{nilpotentsemisimple} that the roots of $\ad(\X)$ are not all purely imaginary implies that the Lie algebras of all the factors of $G$ (that is, $G$ itself) contain stable and unstable vectors for the flow of $\X$.

Similarly, Theorem B' is the statement of Theorem \ref{nilpotentnilpotent} for the case of a simple Lie group $G$.  Here, it is clear that $\U_{1} + \U_{2}$ projects nontrivially to each factor.

\section{Acknowledgements}
I am grateful to the referee for valuable comments on the previous draft; to Joseph Conlon, Gopal Prasad, Benjamin Schmidt, and Alejandro Uribe for helpful discussions and input; to Livio Flaminio for helpful correspondence and, in particular, for providing Lemma \ref{flaminiolemma}; and to my advisor, Ralf Spatzier, for suggesting this problem and for giving fruitful advice throughout this process.

\bibliographystyle{amsplain}
\bibliography{bibliography}

\providecommand{\bysame}{\leavevmode\hbox to3em{\hrulefill}\thinspace}
\providecommand{\MR}{\relax\ifhmode\unskip\space\fi MR }
\providecommand{\MRhref}[2]{%
  \href{http://www.ams.org/mathscinet-getitem?mr=#1}{#2}
}
\providecommand{\href}[2]{#2}
\begin{thebibliography}{10}

\bibitem{Anosov}
D.~V. Anosov and V.~V. Solodov, \emph{Hyperbolic sets}, Dynamical systems,
  {IX}, Encyclopaedia Math. Sci., vol.~66, Springer, Berlin, 1995, pp.~10--92.
  \MR{MR1356042}

\bibitem{FK}
Renato Feres and Anatole Katok, \emph{Ergodic theory and dynamics of
  {$G$}-spaces (with special emphasis on rigidity phenomena)}, Handbook of
  dynamical systems, {V}ol.\ 1{A}, North-Holland, Amsterdam, 2002,
  pp.~665--763. \MR{MR1928526 (2003j:37005)}

\bibitem{FF}
Livio Flaminio and Giovanni Forni, \emph{Invariant distributions and time
  averages for horocycle flows}, Duke Math. J. \textbf{119} (2003), no.~3,
  465--526. \MR{MR2003124 (2004g:37039)}

\bibitem{HT}
Roger Howe and Eng-Chye Tan, \emph{Nonabelian harmonic analysis}, Universitext,
  Springer-Verlag, New York, 1992, Applications of ${{\rm{S}}L}(2,{{\bf{R}}})$.
  \MR{MR1151617 (93f:22009)}

\bibitem{KS}
A.~Katok and R.~J. Spatzier, \emph{Subelliptic estimates of polynomial
  differential operators and applications to rigidity of abelian actions},
  Math. Res. Lett. \textbf{1} (1994), no.~2, 193--202. \MR{MR1266758
  (95b:35042)}

\bibitem{KR}
Anatole Katok, \emph{Cocycles, cohomology and combinatorial constructions in
  ergodic theory}, Smooth ergodic theory and its applications ({S}eattle, {WA},
  1999), Proc. Sympos. Pure Math., vol.~69, Amer. Math. Soc., Providence, RI,
  2001, In collaboration with E. A. Robinson, Jr., pp.~107--173. \MR{MR1858535
  (2003a:37010)}

\bibitem{KS2}
Anatole Katok and Ralf~J. Spatzier, \emph{First cohomology of {A}nosov actions
  of higher rank abelian groups and applications to rigidity}, Inst. Hautes
  \'Etudes Sci. Publ. Math. (1994), no.~79, 131--156. \MR{MR1307298
  (96c:58132)}

\bibitem{KM}
D.~Y. Kleinbock and G.~A. Margulis, \emph{Logarithm laws for flows on
  homogeneous spaces}, Invent. Math. \textbf{138} (1999), no.~3, 451--494.
  \MR{MR1719827 (2001i:37046)}

\bibitem{Knapp}
Anthony~W. Knapp, \emph{Representation theory of semisimple groups}, Princeton
  Landmarks in Mathematics, Princeton University Press, Princeton, NJ, 2001, An
  overview based on examples, Reprint of the 1986 original. \MR{MR1880691
  (2002k:22011)}

\bibitem{M2}
David Mieczkowski, \emph{The cohomological equation and representation theory},
  Ph.D. thesis, The Pennsylvania State University, 2006.

\bibitem{M1}
\bysame, \emph{The first cohomology of parabolic actions for some higher-rank
  abelian groups and representation theory}, J. Mod. Dyn. \textbf{1} (2007),
  no.~1, 61--92. \MR{MR2261072 (2007i:22015)}

\bibitem{Shalom}
Yehuda Shalom, \emph{Explicit {K}azhdan constants for representations of
  semisimple and arithmetic groups}, Ann. Inst. Fourier (Grenoble) \textbf{50}
  (2000), no.~3, 833--863. \MR{MR1779896 (2001i:22019)}

\bibitem{Starkov}
Alexander~N. Starkov, \emph{Dynamical systems on homogeneous spaces},
  Translations of Mathematical Monographs, vol. 190, American Mathematical
  Society, Providence, RI, 2000, Translated from the 1999 Russian original by
  the author. \MR{MR1746847 (2001m:37013b)}

\bibitem{Warner}
Garth Warner, \emph{Harmonic analysis on semi-simple {L}ie groups. {I}},
  Springer-Verlag, New York, 1972, Die Grundlehren der mathematischen
  Wissenschaften, Band 188. \MR{MR0498999 (58 \#16979)}

\end{thebibliography}

\end{document}